\documentclass[oneside,reqno,11pt]{amsart}
\usepackage[left=1in,top=1in,right=1in,bottom=1in]{geometry}
\usepackage{amsmath,amsthm,amsfonts,amssymb}

\newtheorem{theorem}{Theorem}[section]
\newtheorem{lemma}[theorem]{Lemma}

\newtheorem{definition}[theorem]{Definition}
\newtheorem{corollary}[theorem]{Corollary}

\newtheorem{remark}[theorem]{Remark}

\newcommand{\fg}{\mathfrak{g}}
\newcommand\RR{{\mathbb{R}}}

\newcommand\NN{\mathbb{ N}}
\newcommand\ZZ{\mathbb{ Z}}

\newcommand\HH{\mathbb{ H}}

\newcommand{\cdual}{*}
\newcommand{\dup}[2]{\langle #1,#2\rangle}
\newcommand{\Co}{\mathrm{Co}}
\newcommand{\supp}{\mathrm{supp}}
\newcommand{\Mid}{\Big|}

\newcommand{\cS}{\mathcal{S}}
\newcommand{\cSg}{\mathcal{S}(\fg)}
\newcommand{\cSG}{\mathcal{S}(G)}
\newcommand{\cSZ}{\mathcal{S}_0 (G)}
\newcommand{\cSd}{\mathcal{S}^*(G)}
\newcommand{\cSZd}{\mathcal{S}_0^*(G)}
\newcommand{\cP}{\mathcal{P}(G)}

\newcommand{\sL}{\mathcal{L}}

\begin{document}
\title{Coorbit description and atomic decomposition of Besov spaces }
\author{Jens G. Christensen}
\address{Tufts University, Department of Mathematics, 
503 Boston Avenue, Medford, MA 02155}
\thanks{The research of J. G. Christensen was partially supported by
NSF grant DMS-0801010, and ONR grants NAVY.N0001409103, NAVY.N000140910324}
\email{jens.christensen@tufts.edu}

\author{Azita Mayeli}
\address{Department of Mathematics and Computer Sciences, City University of New York (CUNY), Queensborough College,
222-05 56th Avenue Bayside, NY 11364}\thanks{The research of A. Mayeli was partially supported by
NSF grant DMS-0801010}
\email{amayeli@qcc.cuny.edu}

\author{Gestur {\'O}lafsson}
\address{Department of Mathematics, Louisiana State University, Baton Rouge, LA 70803}
\email{olafsson@math.lsu.edu}
\thanks{The research of G. {\'O}lafsson was supported by DMS-0801010 and DMS-1101337} 

\begin{abstract} Function spaces are central topic in analysis. Often those spaces and related analysis involves symmetries in form of an action of a Lie group. Coorbit theory as introduced by Feichtinger and Gr\"ochenig and then later extended in \cite{Christensen2011} gives a
unified method to construct Banach spaces of functions based on
representations of Lie groups. In this article we identify the
homogeneous Besov spaces  on stratified Lie groups introduced in
\cite{FuMa} as coorbit spaces in the sense of \cite{Christensen2011}
and use this to derive atomic decompositions for the Besov spaces.
   \end{abstract}
\maketitle
\section*{Introduction}
\noindent
Banach spaces of functions are central topics in analysis on $\RR^n$. Those spaces are usually translation invariant, a property that can be expressed as invariance under the Lie group $\RR^n$ acting on itself. More generally, one considers Banach spaces defined on a manifold $X$ and invariant under an action of a Lie group $G$. This simple idea combines two central topics in mathematics: Harmonic analysis on manifolds and representation theory of Lie groups. This idea was exemplified in the fundamental
construction of Feichtinger and Gr\"ochenig  \cite{Feichtinger1988,Feichtinger1989a,Feichtinger1989b} where the authors proposed 
a unified way to construct function spaces based on representation theory. This method was generalized in \cite{Christensen2009} by removing
some assumptions on the representation 
from the work of Feichtinger and Gr\"ochenig, and
a method for atomic decompositions of these spaces was developed in
\cite{Christensen2010}. Similar techniques for decomposition
can be found in for example \cite{Grochenig1991} and \cite{Fornasier2005}, 
but the use of differentiable representations in \cite{Christensen2010} 
makes those results particularly convenient for our purposes.

A different way to introduce function spaces is to use smoothness
conditions. As 
important examples of that are the well known
  Besov and and Triebel-Lizorkin-type spaces. Often those 
spaces are described using Littlewood-Paley methods \cite{FJW}, but they
can also be described using wavelet type decompositions, 
see \cite{HL} and the reference therein. Stratified Lie groups are
natural objects for extending this line of study, since they 
come with natural dilations and a sub-Laplacian satisfying 
the H\"ormander conditions.
 
 The inhomogeneous  Besov spaces on stratified Lie groups was introduced by Saka \cite{Saka79}. A characterization of inhomogeneous Besov groups in terms of Littlewood-Paley-decomposition was shown for all groups of polynomial growth in \cite{FMV}. The construction of the homogeneous Besov spaces was extended to stratified 
Lie groups in \cite{FuMa}
  using a discrete collection
 of smooth band-limited wavelets  with vanishing moments 
of all orders.  These wavelets were introduced in \cite{gm1}. 
In \cite{FuMa}, it was also shown that the definition is equivalent with a definition in terms of the heat semigroup
$\{e^{-t\sL}\}_{t>0}$ for $\sL$, the sub-Laplacian operator for the group $G$.
In this paper we extend this result and  show that the definition of Besov spaces  also coincides with a definition based on 
 a family of
operators $\{\widehat\varphi(t\sL)\}_{t>0}$ where $\widehat\varphi\in
C_c^\infty(\RR^+)$ (Theorem \ref{norm-equi-general}).
  We further use a representation theory approach  
  to show that the homogeneous 
Besov spaces on stratified Lie groups in \cite{FuMa} are coorbit spaces in the sense of
\cite{Christensen2009,Christensen2011} and to use the results 
of \cite{Christensen2010} to obtain atomic decompositions and frames
for those spaces.

The article is organized as follows. In Section \ref{Se-Coorbit} we recall the definition of coorbit spaces and their atomic decomposition as needed for the theory of Besov spaces. The details and proofs can be found in \cite{Christensen2009,Christensen2010,Christensen2011}. We start Section \ref{besov-spaces} by recalling the basic definitions for stratified Lie groups. In particular we introduce the sub-Laplacian and
its spectral theory. This enables us to formulate the needed results from \cite{FuMa}. In particular we recall in Theorem
\ref{equivalency-with-heat-kernel} some properties of the homogeneous Besov spaces  and their descriptions in terms of the heat kernel:
For any
   $k\in \NN$, $|s|<2k$,  and any $f\in \cSZd$
   \begin{align*}
     \| f \|_{{\dot B}_{p,q}^s} \asymp
     \left( \int_0^\infty t^{-sq/2 } \|(t\sL)^ke^{-t\sL}f \|_p^q \frac{dt}{t}
     \right)^{1/q}
   \end{align*}
where $\mathcal{L}$ is the sub-Laplacian (\cite{FuMa}, Theorem 4.4).
This description is not well adapted to the coorbit theory. For that
we need to replace the heat kernel   by
a smooth bandlimited function $u$. This is achieved in 
Theorem \ref{norm-equi-general}, which is the fundamental step in describing the
Besov spaces as coorbit spaces and their atomic decomposition in Theorem \ref{analyze vector}: 
  Let $1\leq p,q\leq \infty$ and $s\in\RR$. Let
  $\widehat u\in \cS(\RR^+)$ be compactly
  supported and satisfying some extra conditions. If $u$
  denotes the distribution kernel of $\widehat u(\sL)$,
  then $u$ is an analyzing  vector and up to norm equivalence
  $$B_{p,q}^{Q-2s/q}(G)=\mathrm{Co}_{\cSZ}^u L^{p,q}_s\, .$$
Here $Q$ is the homogeneous degree of the stratified Lie group $G$.

\section{Coorbit spaces and their atomic decomposition}\label{Se-Coorbit}
\noindent
In this section we introduce the notion of coorbit spaces based on \cite{Christensen2009,Christensen2011}.
We then discuss the
discretization from \cite{Christensen2010}
of reproducing kernel Banach spaces and apply it to coorbit spaces.

\subsection{Construction}
The starting point is a continuous representation $\pi$ of
a locally compact group $G$ on a Fr\'echet space $S$. In most applications
the space $S$ has a natural candidate.
Let $S^\cdual$ be the
conjugate linear dual equipped with the weak* topology and
assume that $S$ is continuously embedded and weak*
dense in $S^\cdual$. The conjugate linear
dual pairing of elements $v\in S$ and $\phi\in S^\cdual$ will be denoted
by $\dup{\phi}{v}$.
In particular
$\dup{u}{v}\in \mathbb{C}$ is well defined for $u,v\in S$ and if
$\dup{u}{v}=0$ for all $v\in S$ then $u=0$.

A vector $v\in S$ is called \emph{cyclic} if
$\dup{\phi}{\pi(x)v}=0$ for all $x\in G$ means that
$\phi=0$ in $S^*$, and if such a vector
exists $\pi$ is called a \emph{cyclic representation}.
Define the contragradient
representation $(\pi^\cdual,S^\cdual)$ by
\begin{equation*}
  \dup{\pi^\cdual(x)\phi}{v}:=\dup{\phi}{\pi(x^{-1})v}.
\end{equation*}
Then $\pi^*$ is a continuous representation of $G$ on $S^\cdual$.

For a fixed vector
$u\in S$ the \emph{wavelet transform} $W_u:S^*\to C(G)$
\begin{equation*}
  W_u(\phi)(x) := \dup{\phi}{\pi(x)u} = \dup{\pi^*(x^{-1})\phi}{u}.
\end{equation*}
is a linear mapping. Note that $W_u$ is injective if and only if $u$ is cyclic.

Denote by $\ell_x$ and $r_x$ the left and right
translations on functions on $G$ given by
\begin{equation*}
  \ell_x f(y) := f(x^{-1}y)
  \qquad\text{and}\qquad
  r_xf(y) := f(yx).
\end{equation*}
A Banach space of functions $B$ is called left
invariant if $f\in B$ implies that $\ell_x f\in B$ for
all $x\in G$ and there is a constant
$C_x$ such that $\| \ell_x f\|_B\leq C_x \| f\|_B$
for all $f\in B$. Define right invariance similarly.
We will always assume that for each non-empty compact set $U$ there is a
constant $C_U$ such that for all $f\in B$
\begin{equation*}
  \sup_{y\in U} \| \ell_yf \|_B \leq C_U \| f\|_B
  \qquad\text{and}\qquad
  \sup_{y\in U} \| r_yf \|_B \leq C_U \| f\|_B.
\end{equation*}
If $x\mapsto \ell_xf$ or $x\mapsto r_xf$ are continuous from $G$ to $B$
for all $f\in B$, then we say that left or
right translation are continuous respectively.

Fix a Haar measure $\mu$ on $G$. We will only work with spaces $B$ of functions on $G$
for which convergence in $B$ implies convergence (locally) in left
Haar measure on $G$. Examples of such spaces are all spaces continuously included in some $L^p(G)$ for $1\leq p \leq \infty$.
When $f,g$ are measurable functions on $G$ for which
the product $f(x)g(x^{-1}y)$ is integrable for all $y\in G$
we define the convolution $f*g$ as
\begin{equation*}
  f*g(y) := \int_G f(x)g(x^{-1}y)\,d\mu(x).
\end{equation*}
For a function $f$ on $G$ we define ${f}^\vee(x) := f(x^{-1})$ and
$f^* := \overline{{f}^\vee}$. We will frequently use that for
$f \in {\rm L}^1(G)$, the adjoint of the convolution
operator $g \mapsto g \ast f$ is provided by $h \mapsto h \ast f^*$.

For a given representation $(\pi,S)$ and a vector $u\in S$
define the space
$$B_u := \{ f\in B\, |\, f=f*W_u(u) \}$$
with norm inherited from $B$.
Furthermore define the space
$$\Co_S^u B := \{ \phi\in S^*\, |\, W_u(\phi)\in B  \}$$
with norm $\| \phi \| = \| W_u(\phi)\|_B$.
In \cite{Christensen2011} minimal conditions were listed to ensure
that a $B_u$ and $\Co_S^u B$ are isometrically
isomorphic Banach spaces. In that case, the space $\Co_S^u B$ is called
the \emph{coorbit space of $B$
with respect to $u$ and $(\pi,S)$}.
In this paper we
only need the following result which can be
found in \cite{Christensen2011}.

\begin{theorem}\label{thm:coorbitsduality}
Let $\pi$ be a cyclic representation of a group $G$ on a
Fr\'echet space $S$ which is continuously included in
its conjugate dual $S^*$. Fix a cyclic vector $u\in S$
and assume that
$W_u(\phi )\ast W_u(u)=W_u(\phi )$ for all $\phi \in S^*$.
If for the Banach function space $B$ the mapping 
\begin{equation}\label{eq-doubleCont}
B\times S\to \mathbb{C}\, , \quad (f,v)\mapsto \int_G f(x) W_v(u)^\vee (x)\, d\mu (x)
\end{equation}
is well defined and continuous, then
\begin{enumerate}
  \item The space $\Co_S^u B$ is a $\pi^\cdual$-invariant Banach
    space.     \label{prop2}
  \item The space $B_u$ is a left invariant reproducing kernel
    Banach subspace of $B$.
  \item $W_u:\Co_S^u B\to B_u$ is an isometric isomorphism which
    intertwines $\pi^\cdual$ and left translation. \label{prop3}
  \item If left translation is continuous on $B,$ then
    $\pi^*$ acts continuously on $\Co_S^u B$.\label{prop4}
  \item $\Co_S^u B = \{ \pi^\cdual(f)u \mid f\in
    B_u\}$. \label{prop5}
\end{enumerate}
\end{theorem}

A vector $u$ satisfying the requirements of the theorem
is called an \emph{analyzing vector}.

\begin{remark} Note that the condition (\ref{eq-doubleCont}) implies,
in particular, that for a fixed $u\in S$ we have
$\pi^*(f)\in S^*$  for all $f\in B$. Also, if $S$ is dense in a Hilbert space $H$ and $\pi$ extends to an unitary
representation, also denoted by $\pi$, of $G$ on $H$ such that for $u,v\in S$ (or $H$) $\dup{u}{v}=(u,v)$, then (\ref{eq-doubleCont})
says that
$$(f,v)\mapsto (f,W_u(v))$$
is continuous, where the $(\cdot,\cdot )$ on the right refer to the K\"othe
dual pairing on $B$.
\end{remark}

\subsection{Atomic decompositions and frames}

The previous section sets up a correspondence between
a space of distributions and a reproducing kernel
Banach space. We now investigate the discretization
operators introduced in \cite{Feichtinger1989a}
and \cite{Grochenig1991}, but we do so without assuming integrability
of the reproducing kernel.
The reproducing kernel Banach space $B_u$ has
reproducing kernel $W_u(u)(y^{-1}x)$ and
is continuously included in an ambient Banach space $B$.
 From now on we will 
assume that the space $B$ is solid, which means that
if $g\in B$ and $|f (x) | \leq |g(x)|$ $\mu$-almost everywhere then $f \in B$
and $\| f\|_B \leq \| g\|_B$.

Frames and atomic decompositions for Banach spaces were first introduced
in \cite{Grochenig1991} in the following manner.

\begin{definition}
  \label{def:1}
  Let $B$ be a Banach space and $B^\#$ a Banach sequence
  space with index set $I$. If for $\lambda_i\in B^*$ and $\phi_i \in B$
  we have
  \begin{enumerate}
  \item $\{\lambda_i(f) \}_{i\in I} \in B^\#$ for all $f\in B$;
  \item the norms
    $\| \lambda_i(f)\|_{B^\#}$ and $\| f\|_B$ are equivalent;
  \item $f$ can be written as $f = \sum_{i\in I} \lambda_i(f) \phi_i$;
  \end{enumerate}
  then $\{(\lambda_i,\phi_i) \}_{i\in I}$ is an atomic decomposition
  of $B$ with respect to $B^\#$.
\end{definition}

More generally a Banach frame for a Banach space can be defined as:
\begin{definition}
  \label{def:3}
  Let $B$ be a Banach space and $B^\#$ a Banach sequence
  space with index set $I$. If for $\lambda_i\in B^*$
  we have
  \begin{enumerate}
  \item $\{\lambda_i(f) \}_{i\in I} \in B^\#$ for all $f\in B$,
  \item the norms
    $\| \lambda_i(f)\|_{B^\#}$ and $\| f\|_B$ are equivalent,
  \item there is a bounded reconstruction
    operator $T:B^\# \to B$
    such that $T(\{\lambda_i(f) \}_{i\in I}) = f$,
  \end{enumerate}
  then $\{ \lambda_i \}_{i\in I}$ is an Banach frame for
  $B$ with respect to $B^\#$.
\end{definition}
In the sequel we will often suppress the use of the index set $I$.

The Banach sequence spaces we will use were introduced in
\cite{Feichtinger1989a} and they are constructed from
a solid Banach function space $B$ as described below.
For a compact neighbourhood $U$ of the identity we call
the sequence $\{ x_i\}_{i\in I}$ $U$-relatively separated
if $G\subseteq \bigcup_i x_iU$ and there exists $N\in \mathbb{N}$ such that
\begin{equation*}
  \sup_i (\# \{ j \mid x_iU\cap x_jU \neq \emptyset \}) \leq N.
\end{equation*}
For a $U$-relatively separated sequence $X=\{ x_i\}_{i\in I}$
define the space $B^\#(X)$ to be the collection
of sequences $\{ \lambda_i\}_{i\in I}$ for which
the pointwise sum
\begin{equation*}
  \sum_{i\in I} |\lambda_i| 1_{x_iU}(x)
\end{equation*}
defines a function in $B$.
If the compactly supported continuous functions are dense in $B$ then
this sum also converges in norm (for an argument see \cite{Rauhut2005}
Lemma 4.3.1).
Equipped with the norm
$$\| \{\lambda_i \}\|_{B^\#} := \|
\sum_{i\in I} |\lambda_i| 1_{x_iU} \|_B$$
the space  $B^\#(X)$ is a solid
Banach sequence space, i.e., if $\tau=\{\tau_i\}$ and $\lambda=\{\lambda_i\}$ are sequences such that
$\lambda \in B^\# (X)$ and for all $i\in I$ we have $|\tau_i |\le |\lambda_i|$, then
$\tau \in B^\#(X)$ and $\|\tau \|\le \|\lambda\|$. In the case were $B=L^p(G)$ we have $B^\#(X)=\ell^p(I)$.

For fixed $X=\{ x_i\}_{i\in I}$ the space $B^\#(X)$ only depends
on the compact neighborhood $U$ up to norm equivalence.
Further, if $X=\{ x_i \}_{i\in I}$ and $Y=\{ y_i \}_{i\in I}$ are
two $U$-relatively separated sequences
with same index set such that $x_i^{-1}y_i \in V$
for some compact set $V$, then $B^\#(X)=B^\#(Y)$ with equivalent norms.
For these properties consult Lemma 3.5 in \cite{Feichtinger1989a}.

For a given compact neighbourhood $U$ of the identity
a sequence of non-negative functions $\{ \psi_i\}$ is called a 
\emph{bounded uniform partition of unity} subordinate to $U$ (or $U$-BUPU), if
there is a $U$-relatively separated
sequence $\{ x_i\}$, such that
$\supp(\psi_i)\subseteq x_i U$ and for all $x\in G$ we have
$\sum_i \psi_i (x) =1$. Note that for a given $x\in G$ this sum is over
finite indices.

\subsection{Discretization operators for
  reproducing kernel Banach spaces on Lie groups}

Let $G$ be Lie group with Lie algebra $\mathfrak{g}$
and exponential map $\exp:\mathfrak{g}\to G$. Then for
$X\in\mathfrak{g}$ we define the right and left
differential operators (if the limits exist)
\begin{equation*}
  R(X)f(x) := \lim_{t\to 0} \frac{f(x\exp(tX))-f(x)}{t}
  \qquad\text{and}\qquad
  L(X)f(x) := \lim_{t\to 0} \frac{f(\exp(-tX)x)-f(x)}{t}.
\end{equation*}
We note that $R([X,Y])=R(X)R(Y)-R(Y)R(X)$ and $L([X,Y])=L(X)L(Y)-L(Y)L(X)$.
Fix a basis $\{ X_i\}_{i=1}^{\mathrm{dim}(G)}$
for $\mathfrak{g}$. For a multi index $\alpha$ of
length $|\alpha| =k$
with entries between $1$ and $\mathrm{dim}(G)$
we introduce the operator $R^\alpha$ of subsequent right differentiations
\begin{equation*}
  R^{\alpha} f
  := R(X_{\alpha(k)})  R(X_{\alpha(k-1)}) \cdots R(X_{\alpha(1)}) f.
\end{equation*}
Similarly we introduce the operator $L^\alpha$ of subsequent
left differentiations
\begin{equation*}
  L^{\alpha} f
  := L(X_{\alpha(k)})  L(X_{\alpha(k-1)}) \cdots L(X_{\alpha(1)}) f.
\end{equation*}
We call the function $f$ right (respectively left) \emph{differentiable of order $n$},
if for every $x$ and all $|\alpha|\leq n$ the derivatives
$R^\alpha f(x)$ (respectively $L^\alpha f(x)$) exist.

In the remainder of this section we will assume that
\begin{enumerate}
\item $B$ is a solid left- and right-invariant Banach function
  space on which right translation
  is continuous.
\item The kernel $\Phi$ is smooth and
  satisfies $\Phi*\Phi = \Phi$.
\item The mappings $f\mapsto f*|L^\alpha \Phi|$ and
  $f\mapsto f*|R^\alpha \Phi|$ are continuous on $B$ for
  all $\alpha$ with $|\alpha |\leq \dim(G)$.
\end{enumerate}
By setting $\alpha=0$
we see that convolution with
$\Phi$ is a continuous projection from $B$ onto
the reproducing kernel Banach space
\begin{equation*}
  B_\Phi := \{ f\in B \mid f*\Phi =f   \}.
\end{equation*}

We now describe sampling theorems and atomic decompositions
for the space $B_\Phi$ based on the smoothness of the (not necessarily
integrable) reproducing kernel $\Phi$.
The results can be found in
\cite{Christensen2010} and build
on techniques from
\cite{Feichtinger1988,Feichtinger1989a,Feichtinger1989b,Grochenig1991},
where it should be mentioned that integrability is required.
For integrable kernels not parametrized by groups
see also \cite{Fornasier2005} and \cite{Dahlke08}.
\begin{lemma}
 The mapping
  \begin{equation*}
    B^\# \ni \{ \lambda_i \}
    \mapsto \sum_i \lambda_i \ell_{x_i}\Phi \in B_\Phi
  \end{equation*}
  is continuous.
  The convergence of the sums above is pointwise, and
  if $C_c(G)$ is dense in $B$
  the convergence is also in norm.
\end{lemma}

From now on we let $U_\epsilon$ denote the set
\begin{equation*}
  U_\epsilon := \left\{
    \exp(t_1X_1)\cdots\exp(t_nX_n)
    \Mid -\epsilon \leq t_k \leq \epsilon ,1\leq k\leq n
  \right\}.
\end{equation*}

\begin{theorem}\label{thm:8}
  With the assumptions on $B$ and $\Phi$ 
  listed above we can choose $\epsilon$ small enough
  that for any $U_\epsilon$-BUPU $\{ \psi_i\}$
  the following three operators are all invertible on $B_\Phi$.
  \begin{align*}
    T_1f &:= \sum_i f(x_i)(\psi_i*\Phi) \\
    T_2f &:= \sum_i \lambda_i(f) \ell_{x_i}\Phi\quad
    \text{with $\lambda_i(f) = \textstyle\int f(x) \psi_i(x)\, dx$} \\
    T_3f &:= \sum_i c_i f(x_i) \ell_{x_i}\Phi
    \quad\text{  with $c_i = \textstyle\int \psi_i(x)\,dx$}.
  \end{align*}
  The convergence of the sums above is pointwise and,
  if $C_c(G)$ is dense in $B$,
  the convergence is also in norm.
  Furthermore
  both $\{ \lambda_i(T^{-1}_2 f),\ell_{x_i}\Phi\}$
  and
  $\{ c_iT_3^{-1}f(x_i),\ell_{x_i}\Phi\}$
  are atomic decompositions of $B_\Phi$ and
  $\{c_i\ell_{x_i}\Phi \}$ is a frame.
\end{theorem}

\subsection{Decomposition of coorbits using smoothness}

To obtain frames and atomic decompositions via Theorem~\ref{thm:8}
we will need to introduce differentiability with respect to $\pi$ and
$\pi^*$.
A vector $u\in S$ is called $\pi$-weakly differentiable in the direction
$X\in\mathfrak{g}$ if there is a vector, denoted by $\pi(X)u$, in $S$
such that for all $\phi\in S^*$
\begin{equation*}
  \dup{\phi}{\pi(X)u}=
  \frac{d}{dt}\Big|_{t=0}\dup{\phi}{\pi(e^{tX})u}.
\end{equation*}
For a fixed basis $\{X_i \}\subseteq \mathfrak{g}$
and for a multi-index $\alpha$ define
$\pi(X^\alpha)u$ (when it makes sense) by
\begin{equation*}
  \dup{\phi}{\pi(X^\alpha)u}:=
  \dup{\phi}{\pi(X_{\alpha(k)}) \pi(X_{\alpha(k-1)}) \cdots \pi(X_{\alpha(1)})u}.
\end{equation*}
We note that if $\Phi(x) = W_u(\phi)(x)$
and $u$ is $\pi$-weakly differentiable, 
then $R^\alpha\Phi (x) = W_{\pi(X^\alpha)u}(u)(x)$. We have similar 
notion for the representation $(\pi^*,S^*)$.
A vector $\phi\in S^*$ is called $\pi^*$-weakly differentiable
in the direction
$X\in\mathfrak{g}$ if there is a
vector denoted $\pi^*(X)\phi \in S^*$
such that for all $v\in S$
\begin{equation*}
  \dup{\pi^*(X)\phi}{v}=
  \frac{d}{dt}\Big|_{t=0}\dup{\pi^*(e^{tX})\phi}{v}.
\end{equation*}
For a multi-index $\alpha$ define
$\pi^*(X^\alpha)\phi$ by (when it makes sense)
\begin{equation*}
  \pi^*(X^\alpha)\phi =
  \pi^*(X_{\alpha(k)}) \pi^*(X_{\alpha(k-1)}) \cdots \pi^*(X_{\alpha(1)})\phi
\end{equation*}
We say that $v\in S$, respectively $\phi\in S^*$, is
$k$-times differentiable if $\pi (X^\alpha )v$, respectively 
$\pi^*(X^\alpha)\phi$, exists for all $\alpha$ with $|\alpha|\le k$.
We note that if $\Phi(x) = W_u(\phi)(x)$ and $\phi$ is $\pi^*$-weakly
differentiable, then $L^\alpha\Phi (x) = W_{u}(\pi^*(X^\alpha) \phi)(x)$.
The following result provides atomic decompositions and frames
for coorbit spaces.
\begin{theorem}
  \label{thm:7}
  Let $B$ be a solid and left and right invariant
  Banach function space for which right translations are continuous.
  Assume that $u \in S\subseteq S^*$ is both
  $\pi$- and $\pi^*$-weakly differentiable and
  satisfies the requirements in
  Theorem~\ref{thm:coorbitsduality}.
  Furthermore assume the mappings
  \[f \mapsto f*W_{\pi(X^\alpha)u}(u)\text{ and } f\mapsto f*W_{u}(\pi^*(X^\alpha)u)\]
  are continuous on $B$
  for $|\alpha|\leq \dim(G)$.
  Then we can choose $\epsilon$ small enough that for any
  $U_\epsilon$-relatively separated set $\{ x_i\}$
  the family
  $\{ \pi(x_i)u\}$ is a frame for $\Co_S^u B$.
  Furthermore the families $\{ \lambda_i\circ T_2^{-1}\circ W_u,\pi(x_i)u \}$
  and $\{ c_i T_3^{-1}\circ W_u,\pi(x_i)u \}$ both
  form atomic decompositions for $\Co_S^u B$.
  In particular $\phi\in \Co_S^u B$ can be reconstructed by
  \begin{align*}
    \phi &=
    W_u^{-1 }T_1^{-1} \left( \sum_i W_u(\phi)(x_i)\psi_i*W_u(u)
    \right) \\
    \phi &= \sum_i \lambda_i(T^{-1}_2 W_u(\phi)) \pi(x_i)u     \\
    \phi &= \sum_i c_i (T^{-1}_3W_u(\phi))(x_i) \pi(x_i)u
  \end{align*}
  The convergence of the last two sums is in $S^*,$
  and with convergence in $\Co_S^u B$ if
  $C_c(G)$ is dense in $B$.
  Here $\{\psi_i \}$ is any $U_\epsilon$-BUPU for which
  $\supp(\psi_i)\subseteq x_iU_\epsilon$.
  $T_1,T_2,T_3$ and $c_i$ and $\lambda_i$ are defined as in
  Theorem~\ref{thm:8}.
\end{theorem}

\section{Besov spaces  on stratified Lie Groups}\label{besov-spaces}
\noindent
In this section we introduce the notion of \emph{stratified Lie groups}. We collect some needed
fundamental work by Folland, Hulanicki and Stein \cite{Folland75,FollandStein82,Hulanicki84}. We then
introduce \emph{homogeneous Besov spaces} for stratified Lie groups following \cite{FuMa}. Our main result in this section
is Theorem \ref{norm-equi-general} which is crucial for coorbit description and atomic decomposition of  the Besov spaces. 

\subsection{Stratified Lie groups}\label{preliminary-and-notations}
 A connected and simply connected   Lie group ${G}$
is called stratified
if its Lie algebra $\mathfrak{g}$ decomposes as a direct sum
\begin{equation*}
\mathfrak{g}= V_1\oplus \cdots \oplus V_m=\bigoplus_{j=1}^m V_j\quad \text{ such that } \quad [V_1,V_k] =V_{k+1} \text{ for }1\leq k\leq m
\end{equation*}
where we set
$V_{m+1}=\{0\}$, see \cite{Folland75}. Note that induction shows that
\begin{equation}\label{eq-ViVj}
[V_i,V_j]\subseteq V_{i+j} \text{ for } i,j\ge 1\, .
\end{equation}
Thus $\fg$ is step $m$ nilpotent.
{}From now on $\fg$ is assumed  stratified with fixed stratification $(V_1,\ldots ,V_m)$.
We identify $G$ with $\mathfrak{g}$ through the exponential map. The multiplication is then given by the
Cambell-Hausdorff formula. Examples of stratified Lie groups are Euclidean spaces $\RR^n$, the
Heisenberg group $\HH^n$, and the upper triangular matrices with ones
on the diagonal.

For $r>0$ define the $\delta_r : \fg\to \fg$ by
$\delta_r (X)=r^k X$ if $x\in V_k$. Equation (\ref{eq-ViVj}) shows that $\delta_r$ is an automorphism of $\fg$. The family $\{\delta_r\}_{r> 0}$ is called the canonical family of dilations for $\mathfrak g$. It gives rise to dilations $\gamma_r : G\to G$ given by
\begin{equation*}
\gamma_r(x)=\exp \circ  \delta_r \circ \log (x)
\end{equation*}
with $\log =\exp^{-1}$.
We note that $\gamma_r$ is a homomorphism, and
following \cite{Folland75} we often write $rx$ for $\gamma_r(x)$.

The Lebesgue measure on $\fg$ defines a left and right Haar measure $\mu_G$ on $G$.
We will often write $dx$ instead of $d\mu_G(x)$. 
  Let the number $Q:= \sum_1^m j(\dim V_j)$ be
the \emph{homogeneous degree} of $G$, then we have
\begin{equation*}
  \int_G f(rx)\,dx = r^{-Q} \int_G f(x)\,dx.
\end{equation*}
For example, the homogenous degrees for $G=\RR^n$ and $\HH^n$ are
$Q=n$ and $Q=2n+2$, respectively.

A homogeneous norm on $G$ is a map $|\cdot |: \ G\rightarrow [0,\infty)$, $x\mapsto |x|$,
which is continuous and smooth, except at zero, and satisfies
the following:
\begin{enumerate}
\item $|x|=0$ if and only if $x=0$,
\item  $|rx|=r |x|$ for all $r>0$, and $x\in G$,
\item $|x|=|x^{-1}|$.
\end{enumerate}
Homogenous norms always exist on stratified Lie groups
\cite{Folland75,FollandStein82}. 
For example, on
the Heisenberg group $\HH^n$ with the underlying manifold $\RR^{2n}\times \RR$, a homogenous norm is given by
\begin{align}
  \notag |(x,t)|=(\sum x_j^4+t^2)^{\frac{1}{4}}
\end{align}
It follows by \cite{Folland75}, Lemma 1.2, that
the balls $\{x\in G\mid |x|\le R\}$ are compact  and by
\cite{Folland75}, Lemma 1.4,
there exists a
constant $C>0$ such that $|xy| \le C(|x|+|y|)$ holds for all $x, y\in G$.

\subsection{The sub-Laplacian on stratified Lie groups}

Let $n=\dim G$, $n_k=\dim V_k$ and let $X_1,\ldots ,X_n$ be a basis
for $\fg$ as before. 
We require that
$X_j\in V_k$ for $n_1+\ldots +n_{k-1}+1\le j\le n_1+\ldots +n_k$. In particular, $X_1,\ldots ,X_{n_1}$ is a basis for
$V_1$. Thus, the left invariant differential operators $R(X_1),\ldots ,R(X_{n_1})$ satisfy the H\"ormander condition
\cite{H67} or \cite{Folland75}. We will use coordinates $(x_1,\ldots ,x_n)$ on $\fg$ with respect to this basis. Note  that those
coordinates also form a global chart for $G$. In particular, the space $\cP$ of polynomial functions on $G$ is just the space
$\mathbb{C}[x_1,\ldots ,x_n]$ of polynomials in the coordinates $(x_1,\ldots ,x_n)$.

Denote by $\mathbb{D}(G)$ the algebra of left invariant differential operators on $G$. Recall that $X\mapsto R(X)$ defines a Lie algebra homomorphism into $\mathbb{D}(G)$ and $R(\fg )$ generates $\mathbb{D}(G)$. Note that $\mathbb{D}(G)$ is contained in the Weyl algebra
$\mathbb{C}[x_1,\ldots ,x_n,\partial_1,\ldots  ,\partial_n]$ of differential operators on $\fg$ with polynomial coefficients.

Denote by
$\cSg$  the usual space of Schwartz
functions on $\fg$ and set $\cSG : = \cSg$. Thus, the
Schwartz functions on $G$  
are the smooth functions on $G$ for which
\[
|f|_N :=
\sup_{|\alpha | \le N, x \in G} (1+|x|)^{N} |R^\alpha  f(x)| < \infty\, \quad \text{ for all } N\in \mathbb{N}\, .\]

The topology on $\cSG$ is defined by the family of seminorms $\{|\, \cdot \, |_N \mid N\in\mathbb{N}\}$, see \cite{FollandStein82}, Section D, Page 35.
 This  topology does not depended on
the stratification $(V_1,\ldots ,V_m)$.

According to our previous notation, the conjugate linear dual of $\cSG$ will be denoted by
$\cSd$ and the conjugate dual pairing is
denoted by $\dup{\cdot}{\cdot}$.

Define the space $\cSZ$ of Schwartz functions with vanishing moments by
\begin{align}\notag \cSZ:=\left\{ f\in \cSG \,\left|\,
    \int_G x^\alpha f(x)d\mu (x)=0~\text{ for all multi-indices}~ \alpha\in
    \NN^n\right.\right\}.
\end{align}
The space $\cSZ$ is a closed subspace of $\cSG$ with the relative
topology, and it forms an ideal under convolution:
$\cSG \ast \cSZ \subset \cSZ$.
The inclusion $\cSZ \hookrightarrow \cSG$ induces a continuous linear map $\cSd \to \cSZd$ with kernel
$\cP$, and thus $\cSZd\simeq \cSd/\cP $ in a canonical way.
For the proof of these facts we refer to \cite{FuMa}.

The sub-Laplacian operator on $G$ is
defined by
\begin{equation*}
\sL:=-\sum_{i=1}^{n_1} R(X_i)^2\in \mathbb{D}(G)\, .
\end{equation*}
Note that $\sL$ depends on the space $V_1$ in the given stratification. As $R(X_1),\ldots ,R(X_{n_1})$ satisfies the H\"ormander condition it follows that
$\sL$ is hypoelliptic. Also, $\sL$ is formally self-adjoint and: for
any $f,g\in C_c^\infty(G)$, $\langle \sL f, g\rangle= \langle f,
\sL g\rangle$ which follows by partial integration.
The closure of $\sL$ is also denoted by $\sL$ and
has domain $\mathcal{D} = \{u \in L^2(G)\mid \sL u \in L^2(G)\}$,
where we use $\sL u$ in the sense of distributions. For more information we refer to  \cite{gm1,Folland75,FollandStein82}.
A linear differential operator $D$ on $G$ is called
homogenous of degree $p$ if $D (f\circ \gamma_a)= a^p (D f)\circ
\gamma_a$ for any $f$. If $X\in V_j$,
then $R(X)$ is homogeneous of
degree $j$.  It follows that $\sL^k$ is homogenous of degree $2k$.
For the spectral theory for unbounded operators we refer to \cite{Rudin}, in particular p. 356--370.
Since the closure of $\sL$ is self-adjoint and positive it follows that  $\sL$ has spectral resolution
\begin{align}
\notag \sL= \int_0^\infty \lambda \, dP_\lambda,
\end{align}
where $dP_\lambda$ is the corresponding projection valued measure.
For $\widehat \phi \in L^\infty(\RR^+)$ 
define the 
commutative integral operator on $L^2(G)$ by
\begin{align*}
\widehat \phi (\sL):=\int_0^\infty 
\widehat \phi (\lambda)\, dP_\lambda.
\end{align*}
The operator norm is
$\|\widehat \phi (\sL)\| = \| \widehat \phi \|_\infty$, and
the correspondence $\widehat\phi \leftrightarrow \widehat\phi (\sL)$ satisfies
\begin{enumerate}
\item $(\widehat \phi \widehat \psi )(\sL)=\widehat \phi (\sL)\widehat \psi (\sL)$;
\item $\widehat \phi (\sL)^* =\overline{\widehat{\phi}}(\sL)$.
\end{enumerate}

\begin{theorem}[\cite{Hulanicki84}]\label{th-Hulan}
Let $\widehat{\phi}\in \cS(\RR^+)$. Then there exists $\phi \in \cSG$ such that
\begin{align}\label{eq-Hulan}
\widehat \phi (\sL)f = f \ast \phi \quad \text{ for all }
\; f\in \cSG~.
\end{align}
\end{theorem}
The function $\phi$ is called the
\emph{distribution kernel} for $\widehat{\phi}$,
and from now on, if $\widehat{\phi}\in \cS (\RR^+)$
then $\phi\in\cSG$ will always denote the distribution kernel for
$\widehat{\phi}(\sL)$.
The following two lemmas are simple applications of Hulanicki's theorem.
\begin{lemma}\label{le_star}
Let $\widehat{\psi}\in\cS (\RR^+)$ with distribution kernel $\psi\in\cSG$. Then
the distribution kernel for $\overline{\widehat{\psi}}$ is $\psi^*$
where $\psi^*(x) := \overline{\psi(x^{-1})}.$
\end{lemma}
\begin{proof} Let $f,g\in \cSG$ and recall that $\overline{\widehat{\psi}}(\sL)=\widehat{\psi}(\sL)^*$. The claim follows now because
\[(f,\widehat{\psi}(\sL)^*g)=(\widehat{\psi}(\sL)f,g)=(f*\psi, g)=(f,g*\psi^*)\, .\qedhere\]
\end{proof}
\begin{lemma}\label{le-LkpZero} Let $p$ be a polynomial function on $G$. Then there exists $k\in\NN$ such that $\sL^kp=0$.
\end{lemma}
\begin{proof} Let $P\in \mathbb{C}[\fg]$ be such that $p(\exp (X))=P(X)$. Define
\[q_X(t_1,\ldots ,t_{n_1}):= \sum_{j=1}^{n_1} P(\log (\exp (X)\exp (t_jX_j)))\, .\]
As $G$ is nilpotent, then there exists $m$ such that $q_X$ is a polynomial of degree $\le m$ for all $X\in \fg$. But then
\[\sL^kp (\exp X)
= 
  \left(
  - \sum_{j=1}^{n_1}\dfrac{\partial^2}{\partial t_j^2}
\right)^k
q_X(0)=0\]
for all $k> m/2$ and all  $X\in\fg$.
\end{proof}

If $\widehat{\eta}\in \cS (\RR^+)$ and $\eta$ is the corresponding
kernel function, then we say that $\eta$ is \emph{band-limited} if $\widehat{\eta}$ is compactly supported. In
that case  
there exist  $\epsilon, \delta>0$ such that $\supp (\widehat{\eta})\subseteq [\epsilon, \delta]$.

 \begin{lemma}\label{all-vanishing-moments} Let  $\widehat \phi \in
  \cS (\RR^+)$ with $\widehat \phi \equiv 0$ in some neighborhood of zero and let $\epsilon,\delta>0$. Then
the distribution kernel $\phi $ of $\widehat \phi (\sL)$ lies in $\cSZ$,
and for each $m\in\NN$ there
exists $\widehat \phi_m\in \cS (\RR^+)$ such that $\phi =\sL^m\phi_m$. Furthermore, $\supp (\widehat \phi_m) \subseteq [\epsilon ,\delta]$ if and only if $\supp (\widehat \phi ) \subseteq
 [\epsilon ,\delta ]$.
   \end{lemma}

\begin{proof} Define $\widehat{\phi}_m (s) := s^{-m}\widehat{\phi}(s)$. Then  $\widehat{\phi}(\sL)=\sL^m\widehat\phi_m(\sL)$ and
by Theorem \ref{th-Hulan} and our assumption
on $\phi$
we have $\phi_m\in \cSG$ for all $m\in\NN$    and $\phi=\sL^m \phi_m$. Let $p$ be any polynomial  on $G$ of degree $k\leq 2m-1$. Then
 $\sL^mp=0$ and we have
\[
\int_G p(x) \phi (x)\, d\mu (x) = \int_G p(x)\sL^m\phi_m(x)\, d\mu (x)=\int_G \sL^mp(x)\phi_m(x)\, d\mu (x)=0\]
Thus $\phi $ has vanishing moments of arbitrary order. The last statement is obvious.
\end{proof}

If $\widehat{\varphi}\in \cS(\RR^+)$ and 
$\varphi$ has vanishing moments of all orders, then
$p\ast \varphi=0$ for all $p\in \cP$. Hence
$(f+\cP)\mapsto f*\varphi$ is a well defined linear map on
$\cSZd$.

\subsection{Besov spaces}

To define Besov spaces on an abstract stratified group $G$ we begin as
follows: Pick a multiplier $\widehat\psi\in C_c^\infty(\RR^+)$ with
support in $[2^{-2}, 2^{2}]$.  Denote by $\psi$  the distribution
kernel of the operator $\widehat\psi(\sL)$  as before.
Define
$\widehat\psi_j(\lambda ) := \widehat\psi(2^{-2j}\lambda)$. We assume that
 \begin{align}\label{unity-partition} 
\sum_{j\in \ZZ}
|\widehat\psi_j(\lambda)|^2=1~~~ a.e ~~ \lambda \in \RR^+ ~.
 \end{align}
\begin{lemma}\label{eq-psij} 
For $\widehat{\psi}\in\cS(\RR^+)$ and $r>0$. The distribution kernel $\psi^r$ corresponding
 to the operator $\widehat{\psi}(r\sL)$ is $\psi^r(x)= r^{-Q/2}\psi (r^{-1/2}x)$.
In particular,  let  $\psi_j\in \cS(G)$
denote the distribution kernel of the
operator $\widehat\psi_j(\sL) $. Then $
\psi_j (x)= 2^{jQ}\psi(2^jx)$.
\end{lemma}
This follows easily from the homogeneity of $\sL$. The result is also
true for bounded $\widehat\psi$ on $\RR^+$. For a proof see
\cite{FollandStein82} Lemma 6.29. 
Equation (\ref{unity-partition}) implies  that
\begin{align*}
\sum_{j\in \ZZ}
\widehat\psi_j(\sL)^\ast \circ \widehat\psi_j(\sL)=I~,
\end{align*}
where the sum converges in the strong $L^2$-operator topology.
By (\ref{eq-Hulan}) we therefore get  the following Calder\'{o}n decomposition for all $g\in
L^2(G)$:
\begin{align}\label{decomp-in-L2} 
g=\sum_{j\in \ZZ} g\ast \psi_j \ast
\psi_j^*
\end{align} where the series converges unconditionally in
$L^2$-norm. 
For $g\in \cS_0(G)$,
the expansion (\ref{decomp-in-L2}) also converges in $\cSZ$ in the Schwartz space topology (\cite{FuMa}, Lemma 3.7). Therefore by  duality, for $\cSZd$ in the weak $^*$ topology. 

Assume that (\ref{unity-partition}) holds,
and define the \emph{homogeneous Besov space}
${\dot B}_{p,q}^{s}(G):={\dot B}_{p,q}^{s}$ for $s\in \RR$, $1\leq p \le \infty$
and $1 \le q < \infty$,
to be the set of all $f\in \cSZd$ for
which
 \begin{align}\label{norm-besov-space} \|f\|_{{\dot B}_{p,q}^{s}}:=
\left( \sum_{j\in \ZZ} 2^{jsq} \| f \ast \psi_j \|_p^q \right)^{1/q}<
\infty~,
 \end{align} with
 the natural convention for when $p=\infty$ or $q=\infty$. (Note that  the convolution with a distribution is understood in the weak sense.)
 Here we collect some of   properties of the  homogeneous Besov spaces
in the following lemma. For proofs we refer to \cite{FuMa}:

 \begin{lemma}\label{equivalency-with-heat-kernel} The following holds:
\begin{enumerate}
\item The Besov spaces ${\dot B}_{p,q}^{s}$ with norm  $ \|\cdot\|_{{\dot B}_{p,q}^{s}}$ are Banach spaces.
 \item  For all $1 \le p,q \le \infty$ and all $s \in \RR$, one has
  continuous inclusion maps $\cSZ \hookrightarrow {\dot
    B}_{p,q}^s \hookrightarrow \cSZd$. And for any $1\leq p,q<\infty$ and   $k\in \NN_0$,  $\sL^k\cSZ$  is  dense in
${\dot B}_{p,q}^s$.
\item   For any given $1 \le p,q < \infty$ the
  decomposition (\ref{decomp-in-L2}) converges for all
  $g \in {\dot B}_{p,q}^s$ in the Besov space norm.
\item   For any
   $k\in \NN$, $|s|<2k$, and any $f\in \cSZd$
   \begin{align}
     \label{asympto_rel} \| f \|_{{\dot B}_{p,q}^s} \asymp
     \left( \int_0^\infty t^{-sq/2 } \|(t\sL)^ke^{-t\sL}f \|_p^q \frac{dt}{t}
     \right)^{1/q}~.
   \end{align}
\end{enumerate}
 \end{lemma}

We note that  the corresponding kernel function to 
 $\widehat \psi (\lambda) =(\lambda )^k e^{-\lambda}$ is \textbf{not} in $\cSZ$ so we can not use (\ref{asympto_rel})
 for our purposes. Our aim is therefore to show that we can replace this function with a bandlimited function.

\subsection{Equivalent norms on Besov spaces}
The chief result of this section
is stated in Theorem \ref{norm-equi-general} where the Besov norm (\ref{norm-besov-space}) is described in terms of 
 band-limited  kernel functions with vanishing moments of all orders.
 
In the sequel, we will use the notation ``$\preceq$'' to mean ``$\leq$''
up to a positive constant.

\begin{lemma}\label{exchang-dilation-parameters} Let $
\widehat\varphi,~ \widehat\psi \in \cS(\RR^+)$. Then for any $r, s>0$  and  $f\in L^2(G)$
 \begin{align*} \widehat\varphi(r\sL)\widehat\psi(s\sL)f=
\widehat\varphi(rs^{-1}\sL)\widehat\psi(\sL)f= \widehat\varphi(\sL)\widehat\psi(sr^{-1}\sL)f ~.
 \end{align*}
 \end{lemma}

\begin{proof} This follows easily from Lemma \ref{eq-psij} and  the equation (\ref{eq-Hulan}).\end{proof}

  \begin{lemma}\label{exponential estimation theorem} Let $\widehat
\varphi, \widehat\psi\in C_c^\infty(\RR^+)$ and $m\in\ZZ$. Then there exists a constant
 $c=c(m,  \varphi, \psi)$ such that
  $$
  \|  \widehat\varphi(r\sL) \psi \|_1 \leq  c r^m\quad \forall ~r>0\, .
  $$
 \end{lemma}

 \begin{proof} By Lemma  \ref{all-vanishing-moments}  we have $\psi=\sL^m \psi_m$ for some
$\psi_m\in \cSZ$. Denote the distribution kernel  of $\widehat{\varphi}(r\sL)$ by
$\varphi^r\in \cSZ$. Then by the equation (\ref{eq-Hulan})
\begin{equation*}
  \| \widehat\varphi (r \sL)
\psi\|_1= \| \widehat\varphi(r \sL) \sL^m\psi_m\|_1  =  \| (\sL^m\psi_m)\ast
\varphi^r \|_1 = \|\psi_m\ast (\sL^m \varphi^r)\|_1\, .
\end{equation*}

Now, Lemma \ref{eq-psij},  the homogeneity of $\sL^m$, and  Young's inequality for $p=q=1$
implies

\[ \| \widehat\varphi (r \sL)
\psi\|_1\leq \| \sL^m \varphi^r \|_1~\|\psi_m
\|_1\leq r^{m} \| \sL^{m} \varphi \|_1~
\|\psi_m  \|_1= c \ r^{m}\]
with $c= \| \sL^{m}\varphi\|_1~ \|\psi_m
\|_1$.
\end{proof}

The following is the main result of this section.
\begin{theorem}\label{norm-equi-general} Let $ \widehat\varphi\in \cS(\RR^+)$ with compact support.  Then
for any $f\in \cSZd$

 \begin{align}\label{asympto-relation} \| f \|_{{\dot B}_{p,q}^s}
\asymp \left( \int_0^\infty t^{-sq/2 } \|\widehat\varphi(t\sL)f \|_p^q
\frac{dt}{t} \right)^{1/q}~.
\end{align}
\end{theorem}

 \begin{proof}  Fix $1\leq p\leq \infty$. We first show that  ``$\succeq$" holds.  Let
 $0\not = f\in \cS_0^*(G)$
 and assume that the right hand side in (\ref{asympto-relation}) is finite.
 We can rewrite the integral 
 as follows. 

\begin{align*}
\int_0^\infty t^{-sq/2 }
\|\widehat\varphi(t\sL)f \|_p^q \frac{dt}{t} = \int_1^4 \sum_{l\in \ZZ}
(2^{2l}t)^{-sq/2} \|\widehat\varphi(2^{2l}t\sL)f\|_p^q ~\frac{dt}{t} ~.
\end{align*}
As $\widehat\varphi(2^{2l}t\sL)f$ is in $\cS_0^*(G)$ and that  
    \eqref{decomp-in-L2} converges in $\cS_0^*(G)$ in weak $^*$ topology,  we can write 
\begin{equation}\label{Young-Ineq-1}
  \|\widehat\varphi(2^{2l}t\sL)f\|_p=\left\| \sum_{j\in\ZZ}
    \widehat\varphi(2^{2l}t\sL)f\ast \psi_j \ast \psi_j^\ast \right\|_p \leq
  \sum_{j\in\ZZ} \| \widehat\varphi(2^{2l}t\sL) f\ast \psi_j \ast \psi_j^\ast \|_p~.
\end{equation}
By the left invariant property  of  the operator $\sL$ and hence $\widehat\varphi(2^{2l}t\sL)$, along with an application of Young's inequality
for $p\geq 1$ and Lemma \ref{exchang-dilation-parameters}   we get   
\begin{align}\notag \sum_j \| \widehat\varphi(2^{2l}t\sL) f\ast \psi_j
  \ast \psi_j^\ast \|_p \leq \sum_j \| \widehat\varphi(2^{2l}t\sL)
  \psi_j^\ast\|_1~ \|f \ast \psi_j \|_p= \sum_j \|
  \widehat\varphi(2^{2(l+j)}t\sL) \psi^\ast\|_1~ \|f\ast \psi_j \|_p ~.
\end{align}
By interfering the preceding in (\ref{Young-Ineq-1}) and
then multiplying both sides by $2^{-ls}$ we obtain the following.
\begin{align}\label{last}
2^{-ls} \| \widehat\varphi(2^{2l}t\sL)f\|_p
  & \leq \sum_j 2^{-ls} \| \widehat\varphi(2^{2(l+j)}t\sL) \psi_j^\ast\|_1 \|f\ast
  \psi_j \|_p\\
 &= \sum_j   \left( 2^{-(l+j)s} \|
  \widehat\varphi(2^{2(l+j)}t\sL) \psi^\ast\|_1\right) \left(2^{js} \|f\ast \psi_j \|_p\right)~.\nonumber
\end{align}
To simplify  the notations above, let $A_j^t:= 2^{js}
\| \widehat\varphi(2^{-2j}t\sL) \psi^\ast\|_1$ and $B_j:= 2^{js} \|f\ast
\psi_j \|_p$.
Using new definitions and the definition of convolution on $\ZZ$,
the summation in the inequality (\ref{last}) is a convolution on $\ZZ$. Therefore the inequality can be rewritten as
\begin{align}\notag
  2^{-ls} \| \widehat\varphi(2^{2l}t\sL)f\|_p &\leq \{B_j\}_j\ast \{A_j^t\}_j(-l).
\end{align}

Taking the sum on $l\in \ZZ$ above and then again applying Young's inequality
yields
\begin{align}\label{double-star}
\sum_j 2^{-lsq} \|
  \widehat\varphi(2^{2l}t\sL)f\|_p^q &\leq \|\{B_j\}\|_q^q \|
  \{A_j^t\}\|_1^q~.
\end{align}
 Integrating both sides of (\ref{double-star}) over $t$ in $[1,4]$,   substituting $ \{A_j^t\}$  back, and  that  $\|\{B_j\}\|_q^q= \|f\|_{B_{p,q}^s}^q$, implies
\begin{align}\label{integral}
  \int_0^\infty t^{-sq/2} \|
  \widehat\varphi(t\sL)f\|_p^q ~\frac{dt}{t}
  \leq  \|f\|_{B_{p,q}^s}^q \int_1^4
  \left(\sum_j 2^{js}\| \widehat\varphi(2^{-2j}t\sL)\psi^*\|_1\right)^q ~\frac{dt}{t}
\end{align}

Therefore to complete the proof it remains to show that the sum in the preceding relation  is finite.
Without loss of generality, we shall prove this only for negative
$j$'s. For positive $j$ the assertion follows with a similar argument.

Let $m\in \ZZ$ with $m<s/2$.   By Lemma \ref{exponential estimation theorem} we have
$$\|  \widehat\varphi(2^{-2j}t\sL) \psi^\ast\|_1 \leq  c (2^{-2j}t)^m
$$
for some constant $c$ independent of $t$ and $j$.
Now it is simple to see that the application of the preceding
estimation in the integral (\ref{integral}) for the summands over
negative $j$ implies the finiteness of the integral (\ref{integral})
for $m<s/2$. For the positive $j$ we take $m>s/2$.

To show  ``$\preceq$'' holds, let $ c=\int_0^\infty
|\widehat\varphi(t)|^2 dt/t$. This is finite by the choice of
$\widehat\varphi$. By dilation invariance property of the measure
$dt/t$, for any $\lambda>0$ we can rewrite $ c=\int_0^\infty
|\widehat\varphi(\lambda t)|^2 ~\frac{dt}{t}$.  For simplicity, we assume that
$c=1$.  By substituting $\lambda\mapsto \sL$ and using the spectral
theory for the sub-Laplacian $\sL$, we derive the following Calder\`on
formula from the integral:
\begin{align}\label{identity-relation} I= \int_0^\infty
  |\widehat\varphi |^2(t\sL) ~\frac{dt}{t} = \int_0^\infty
  \widehat\varphi(t\sL)^* \widehat\varphi(t\sL) ~\frac{dt}{t}
\end{align}
where $I$ is the identity,
and the integral converges
in the strong sense. Applying (\ref{identity-relation}) to $\widehat\psi_j(\sL)f$  
  and employing
 $\widehat\psi_j(\sL)\widehat\varphi(t\sL)=\widehat\varphi(t\sL)\widehat\psi_j(\sL)$, we have
\begin{align}\label{something}
  \widehat\psi_j(\sL)f
  = \int_0^\infty
  \widehat\varphi(t\sL)^*\widehat\psi_j(\sL) \widehat\varphi(t\sL)f ~\frac{dt}{t}
 = \int_1^4 \sum_{l\in\ZZ} \widehat\varphi(2^{2l}t\sL)^*\widehat\psi_j(\sL) 
  \widehat\varphi(2^{2l}t\sL)f ~\frac{dt}{t}
\end{align}

If $\psi_j$ is the distribution kernel of $\widehat\psi_j(\sL)$, then the
distribution kernel of $\widehat\varphi(2^{2l}t\sL)^*\widehat\psi_j(\sL) $ is
$\widehat\varphi(2^{2l}t\sL)^*\psi_j $ and for any
$g\in \cSZd$ we have 
$$ \widehat\varphi(2^{2l}t\sL)^*\widehat\psi_j(\sL)g=g\ast  \widehat\varphi(2^{2l}t\sL)^*\psi_j ~ \quad\quad (\text{in the dual sense}).
$$
In particular, for $g= \widehat\varphi(2^{2l}t\sL) f $,  in
(\ref{something}) we have
\begin{align}\label{something-2} 
\widehat\psi_j(\sL)f = \int_1^4
\sum_{l\in\ZZ} \widehat\varphi(2^{2l}t\sL) f \ast \widehat\varphi(2^{2l}t\sL)^*
\psi_j~\frac{dt}{t}~.
 \end{align}

First, by applying Minkowski's inequality for integrals and series
and then Young's inequality for $p\geq 1$, from  (\ref{something-2})  we
deduce the following.

 \begin{align}\notag
 \|\widehat\psi_j(\sL) f\|_p \leq
\int_1^4 \sum_{l\in\ZZ} \|\widehat\varphi(2^{2l}t\sL) f\|_p ~ \|
\widehat\varphi(2^{2l}t\sL)^* \psi_j\|_1 ~\frac{dt}{t}~.
 \end{align}

Therefore for any $q\geq 1$
 \begin{align} \label{est_single_Delta-1}
 \|\widehat\psi_j(\sL) f\|_p^q \preceq
\int_1^4 \left(\sum_{l\in\ZZ} \|\widehat\varphi(2^{2l}t\sL) f\|_p ~ \|
\widehat\varphi(2^{2l}t\sL)^* \psi_j\|_1 \right)^q~ ~\frac{dt}{t}~.
 \end{align}

If in Lemma \ref{exchang-dilation-parameters}  we let $r=2^{2l}t$ and
$s=2^{-2j}$, then
${\widehat\varphi}(2^{2l}t\sL)^*\widehat\psi(2^{-2j}\sL)= {\widehat\varphi}(2^{2(l+j)}t\sL)^*\widehat\psi(\sL)$
and consequently
  $\widehat\varphi(2^{2l}t\sL)^* \psi_j=\widehat\varphi(2^{2(l+j)}t\sL)^*\psi$.
 Using this  in (\ref{est_single_Delta-1}),  taking sum over $j$  with weights $2^{jsq}$ and applying    Fubini's theorem yields
 \begin{align*} 
\sum_{j\in\ZZ}2^{jsq}
\|\widehat\psi_j(\sL) f\|_p^q & \leq \int_1^4 \sum_{j\in\ZZ}2^{jsq}
\left(\sum_{l\in\ZZ} \|\widehat\varphi(2^{2l}t\sL)
f\|_p~\|{\widehat\varphi}(2^{2(l+j)}t\sL)^* \psi\|_1\right)^q ~\frac{dt}{t}\\\notag & = \int_1^4 \sum_{j\in\ZZ} \left(\sum_{l\in\ZZ} 2^{-ls}
\|\widehat\varphi(2^{2l}t\sL) f\|_p ~ ~ 2^{(j+l)s}
\|{\widehat\varphi}(2^{2(l+j)}t\sL)^* \psi\|_1\right)^q ~\frac{dt}{t} ~.
 \end{align*}
The summation over $l$   is a convolution
of two sequences at $-j$.  Therefore by Young's inequality for $q$
  \begin{align}\label{irgend-etwas} 
\sum_{j\in\ZZ}2^{jsq}
\|\widehat\psi_j(\sL) f \|_p^q \leq \int_1^4 \sum_{l\in\ZZ} 2^{-lsq}
\|\widehat\varphi(2^{2l}t\sL) f\|_p^q ~ ~ \left(\sum_{l\in\ZZ} 2^{-ls}
\|\widehat\varphi(2^{-2l}t\sL)^* \psi\|_1\right)^q ~\frac{dt}{t}~.
  \end{align} Note that the left side of (\ref{irgend-etwas}) is  $\|f\|_{\dot B_{p,q}^s}^q$.  And, by a similar argument used
above, one can show that  the sum $\sum_{l\in\ZZ} 2^{-ls}\|\widehat\varphi(2^{-2l}t\sL)^*
\psi\|_1$ is finite.
With these and that $1\leq t\leq 4$, in
(\ref{irgend-etwas}) we proceed  as below.
    \begin{align*}\|f\|_{\dot B_{p,q}^s}^q&\preceq \int_1^4 \sum_{l\in\ZZ} 2^{-lsq} \|\widehat\varphi(2^{2l}t\sL)
f\|_p^q ~\frac{dt}{t} \\\notag &\preceq \int_1^4 \sum_{l\in\ZZ}
(2^{2l}t)^{-sq/2} \|\widehat\varphi(2^{2l}t\sL) f\|_p^q ~\frac{dt}{t}\\\notag &=
\int_0^\infty t^{-sq/2} \|\widehat\varphi(t\sL) f\|_p^q ~\frac{dt}{t}~.
\end{align*}
This completes the proof  of $``\preceq"$.

 Using the standard convention, the proof of the theorem for the cases
$q=\infty$  and $p=\infty$ is simple and we omit it here.
 \end{proof}
 
 We shall say  a compactly supported function  $\widehat\phi\in \mathcal S(\RR^+)$ is a cut off function on $\RR^+$ if  $\widehat\phi\equiv 1$ in an open subinterval of its support. 
The following corollary generalizes the results of Theorem 4.4 in \cite{FuMa}. 
\begin{corollary} Let $ \widehat\varphi $
  be any cut off function on $\RR^+$. Then for any $\widehat\psi\in \cS(\RR^+)$
  
   \begin{align*}
    \| f \|_{{\dot B}_{p,q}^s} \asymp
    \left( \int_0^\infty  \|\widehat\psi(t\sL)\widehat\varphi (t\sL)f \|_p^q \frac{dt}{t}\right)^{1/q}~ .
  \end{align*}
  In particular one takes $\widehat\psi(\lambda)= \lambda^k e^{-\lambda}$ for any $k\in \NN$.  
     
\end{corollary}

\section{Besov spaces as coorbits}\label{se-BSC}

\subsection{Some representation theory of stratified Lie groups}
In this section we will look at some representation
theoretic results for the semidirect product $\RR^+\ltimes G$
where $\RR^+$ acts on $G$ by dilations.
For any $a>0$ and any function
$f$ on ${G}$, define the
dilation of $f$ as follows
\begin{align*}
D_a f (x)= a^{-Q/2} f(a^{-1}x), 
\end{align*}
 where $Q$ is the homogenous degree of $G$.
Let $\pi$ denote the quasi-regular representation of the
semidirect product $\RR^+\ltimes G$
defined by
\begin{equation}
  \label{eq:repn}
  \pi(a,x)f=\ell_xD_a f, \ ~~ a>0, \ x\in G
\end{equation}
for all $f\in L^2(G)$.
The representation $\pi$ induces a family of infinitesimal operators
$\pi^\infty(X)$ for $X$ in the Lie algebra of $\RR^+\ltimes G$
\begin{equation*}
  \pi^\infty(X)f = \lim_{t\to 0} \frac{\pi(\exp(tX))f-f}{t},
\end{equation*}
with limit in $L^2(G)$. We
use $\pi^\infty$ to separate this from the
weak derivatives we have used before, yet
note that a strongly smooth function in $\cSZ$ is also
both weak and weak$^*$ differentiable.
\begin{lemma}\label{lem:repn}
  $(\pi, \cS_0(G))$ and
  $(\pi^\infty,\cS_0(G))$ are representations
    of $\RR^+\ltimes G$ and its  Lie algebra respectively.
\end{lemma}
\begin{proof}
  By Propositions 1.46  and 1.25 in \cite{FollandStein82} we already
  know that $\cS_0(G)$ is invariant under the left translation.
  That  $\cS_0(G)$ is dilation invariant follows from the
  homogeneity of the norm $|\cdot |$.
  From the seminorms $|\cdot |_N$ it is not
  hard to see that $\cS_0(G)$ is invariant under
  left differentiation, and thus we  only need to
  show that the differentiation arising from dilations
  leaves $\cS_0(G)$ invariant.
  Let $T=(1,0)$ be in the Lie algebra of $\RR^+\ltimes G$, then
  there exists polynomials $p_\alpha$ such that
  \begin{equation*}
     \pi^\infty(T)f(x) 
     = \frac{d}{dt}\Big|_{t=0}\pi(\exp(tT))f(x)
     =\frac{d}{dt}\Big|_{t=0} e^{-t/2} f(e^{-t}x)
    = - \frac{1}{2} f(x) + \sum_{|\alpha|=1} p_\alpha(x) R^\alpha f(x),
  \end{equation*} 
  where $p_\alpha$ are polynomials of degree $1$ in $x$. 
  For this statement see p. 41 in \cite{FollandStein82} coupled with
  p. 25 in \cite{FollandStein82}. 
  This proves invariance under the infinitesimal operator
  $\pi^\infty(T)$ .
\end{proof}

The Haar measure on $\RR^+\ltimes G$ is given by
$d\mu(x,a)=a^{-(Q+1)}\,dx\,da$ where $dx$ is the Haar measure on $G$. 
The following convolution
identity is crucial to our main result.
\begin{lemma}
  \label{lem:reprformula}
  Let $u\in \cSZ$ be the distribution kernel for
 $\widehat{u}$ with support away from zero satisfying
  \begin{equation}\label{eq:1}
    \int_0^\infty \frac{|\widehat{u}(\lambda)|^2}{\lambda} \,d\lambda = 1.
  \end{equation}
  Then with convolution on $\RR^+\ltimes G$ we have
  \begin{equation*}
    W_u(\phi)*W_u(u)(a,x) = W_u(\phi)(a,x)
  \end{equation*}
  for all $\phi\in \cSZd$.
\end{lemma}
\begin{proof} 
   Since $u,v\in\cSZ$ have vanishing moments
  of all orders, then
  $u*(D_a v)(x) = (u+p_1)*D_a(v+p_2)$ for all polynomials $p_1,p_2$.
  In particular if we let $p_x$ be the right Taylor polynomium
  of homogeneous degree $k$ for $u$ around $x$, then
  \begin{equation*}
    \int u(y)v(a^{-1}y^{-1}x)\,dy
    = \int u(xy)v(a^{-1}y^{-1})\,dy
    = \int (u(xy)-p_x(y)) v(a^{-1}y^{-1})\,dy.
   \end{equation*}
   By the esimate of $u(xy)-p_x(y) $ from Lemma 3.2 in \cite{FuMa} we therefore get
   \begin{align}
     |u*D_av(x)|
     &\leq C(u,k) \int a^{-Q/2}|y|^{k+1} |v(a^{-1}y^{-1})|\,dy \notag\\
     &\leq C(u,k) a^{Q/2} \int |ay|^{k+1} |v(y^{-1})|\,dy \notag\\
    &\leq C(u,v,k,m) a^{Q/2+k+1} (1+|x|)^{-m}. \label{eq:4}
  \end{align}
   A similar inequality is obtained in (10) from \cite{FuMa}
  (note the dilation there is different from $D_a$):
  \begin{equation}\label{eq:5}
    |u*D_a v(x)| \leq C(u,v,k,m) a^{-k-1-Q/2} (1+|x|)^{-m}.
  \end{equation}

  With these facts in place we are ready
  to prove the reproducing formula.
  Using the definition of $\pi$, this is the same as showing that
  for all $\phi\in \cSd$, $a\in\RR^+$ and $x\in G$
  \begin{equation*}
    \int_0^\infty
    \phi*(D_bu)^**(D_bu)*(D_au)^*(x)\, \frac{db}{b^{Q+1}}
    = \phi*(D_au)^*(x).
  \end{equation*}
  All convolutions involving $u$ commute (by definition of $u$)
  and thus we need to show
  \begin{equation*}
    \lim_{\epsilon\to 0,N\to \infty}
    \int_\epsilon^N \phi*(D_au)^**(D_bu)*(D_bu)^*(x)\, \frac{db}{b^{Q+1}}
    = \phi*(D_au)^*(x).
  \end{equation*}
  Since $(D_a u) * (D_a u)^* =  a^{-Q/2} D_a(u*u^*)$ the assertion will
  follow if for all $v\in\cSZ$
  \begin{equation*}
    g_{\epsilon,N}(x)
    :=\int_\epsilon^N v*D_b(u*u^*)\, \frac{db}{b^{Q+1}}
 \end{equation*}
  converges in $\cSZ$
  (by the $L^2$ spectral theory and \eqref{eq:1} the limit is $v$).
  For $g_{\epsilon,N}$ to be Cauchy in $\cSZ$ it suffices to show that
  for all $m$ and $|\alpha|\leq m$ we have
  \begin{equation*}
    \int_0^\infty \sup_{x\in G}
    (1+|x|^m) |R^\alpha (v*D_b(u*u^*))(x)|\,\frac{db}{b^{Q+1}}
    < \infty.
  \end{equation*}
  This can be verified  by choosing $k$ and $m$ large enough in \eqref{eq:4}
  (when $b\leq 1$) and \eqref{eq:5} (when $b\geq 1$).
\end{proof}
\begin{remark}
  The integral \eqref{eq:1} is finite for any $\widehat u\in
  \mathcal S(\RR^+)$ with support away zero. 
  Therefore for the lemma one needs to  normalize 
  $\widehat u$  such that the integral is one. 
\end{remark}

  \subsection{Coorbit description and atomic decompositions
for Besov spaces}

In this section we show that the Besov spaces on stratified Lie groups
can be described via certain Banach space norms of  wavelet transforms (see Section 1).
The Banach spaces of interest are
equivalence classes,
$L_{s}^{p,q}:=L^{p,q}_s(\RR^+\times G)$ for $1\leq p,q\leq \infty$, ~$s\in \RR$,
of measurable functions for which
\begin{equation*} \| f\|_{L^{p,q}_s} = \left( \int_{\RR^+} \left(
\int_G |f(a,x)|^p dx \right)^{q/p}a^{-qs/2} \,\frac{da}{a}
\right)^{1/q}~<\infty~.
\end{equation*}

As a corollary to Lemma ~\ref{equivalency-with-heat-kernel} (b)
 and Theorem~\ref{norm-equi-general} we note:
\begin{corollary}\label{w-inequality}
  Let $1\leq p,q \leq \infty$ and $s\in\RR$, then
  for any vector $u\in \cSZ$ the mapping
  $$\cSZ\ni v\mapsto W_u(v)\in L^{p,q}_s$$
  is continuous.
\end{corollary}

The main result of  this paper  follows:

\begin{theorem}\label{analyze vector}
   Let
  $\widehat u\in \cS(\RR^+)$ be compactly
  supported.
  Then $u$ is an analyzing  vector and,   
  for any $1\leq p,q\leq \infty$ and $s\in\RR$, 
     up to norm equivalence we have 
  $$B_{p,q}^{Q-2s/q}(G)=\mathrm{Co}_{\cSZ}^u L^{p,q}_s.$$
  Furthermore, the frames and atomic decompositions from
  Theorem~\ref{thm:7} all apply.
\end{theorem}

\begin{proof}
  We  first verify the conditions of Theorem~\ref{thm:coorbitsduality}
  with $S=\cSZ$, $B = L_s^{p,q}(\RR^+\times G)$, 
  and $\pi$ as given in \eqref{eq:repn}.

  By Lemma~\ref{all-vanishing-moments} it is known
  that $u$ is in $\cSZ$.
  Let $\phi\in \cSZd$ such that
  $\langle \pi(a,x)u,\phi\rangle=0$ for
  all  $(a,x)\in \RR^+\times G$. By Theorem \ref{norm-equi-general}
  $\phi\equiv 0$ in $\dot B_{p,q}^s$ for all $1\leq p,q\leq \infty$ and
  $s\in \RR$. The continuous embedding
  $B_{p,q}^s \hookrightarrow \cSZd$  implies
  that $\phi\equiv 0$ in $\cSZd$. Thus $u$ is a cyclic
  vector in $\cSG$.

  The following H\"older inequality and
  Corollary~\ref{w-inequality} imply the continuity of map (\ref{eq-doubleCont})
  in Theorem~\ref{thm:coorbitsduality}:
  $$\|f\ast W_u(u)\|_{L_{s'}^{\infty,\infty}}
  \leq  \|f\|_{L_{s}^{p,q}} \|W_u(u)\|_{L_{-s}^{p',q'}}
  $$
  where $1/p+1/p'=1$ and $1/q+1/q'=1$,  $s'=s-2Q/q'$.

   The multiplier  $\widehat u$  satisfies \eqref{eq:1} (otherwise we normalize $\widehat u$), therefore the reproducing formula
  $W_u(\phi)*W_u(u) =W_u(\phi)$ holds true by Lemma~\ref{lem:reprformula}. Theorem \ref{norm-equi-general} completes the proof of the first part.

  For the last statement we need to verify   that the conditions of
  Theorem~\ref{thm:7} are satisfied.   The invariance of $\cSZ$ under $\pi^\infty$
  (and hence the invariance under weak derivatives) follows from 
  Lemma~\ref{lem:repn}. 
  By Young's inequality derived below
  \begin{align*}
 \|f\ast g\|_{L_s^{p,q}}&=
 \Big\| \int_{\RR^+}\int_G  \tilde g(a,x)  { R_{(a,x) } f}     ~ a^{-(Q+1)}dadx\Big\|_{ L_{s}^{p,q} }\\
 &\leq  \int_{\RR^+}\int_G   |\tilde g(a,x)|~ \|  { R_{(a,x) } f}  \|_{L_s^{p,q}}~ a^{-(Q+1)}dadx\\
 &= \|f\|_{L_{s}^{p,q}}\|\tilde g\|_{L_{2Q+s}^{1,1}}\\
 &= \|f\|_{L_{s}^{p,q}}\|g\|_{L^{1,1}_{-s} }~, 
\end{align*}
and  Corollary~\ref{w-inequality}, for any 
$|\alpha|\leq \dim(\RR^+\ltimes G)$, 
  all convolutions
 \begin{equation*}
  B\ni f\mapsto f*W_{\pi(X^\alpha)u}(u) \in B
  \quad\text{and}\quad
  B\ni f\mapsto f*W_{u}(\pi^*(X^\alpha)u) \in B
\end{equation*}
are continuous and thus the decompositions apply.
\end{proof}

\end{document}